\documentclass[10pt]{amsart}
\usepackage{ amssymb,latexsym, amscd}
\usepackage[all]{xy}
\usepackage[square, numbers]{natbib}
\usepackage{graphicx}
\usepackage{mathrsfs}
\usepackage{amsmath}

 \newtheorem{theorem}{Theorem}[section]
 \newtheorem{cor}[theorem]{Corollary}
 
 \newtheorem{proposition}[theorem]{Proposition}
 \theoremstyle{definition}
 \newtheorem{definition}[theorem]{Definition}
 \theoremstyle{definition}

 \newtheorem{rem}[theorem]{Remark}
 \numberwithin{equation}{section}

\newcommand{\ben}{\begin{equation}}
\newcommand{\een}{\end{equation}}

\newcommand{\integer}{\ensuremath{{\mathbb Z}}}

\newcommand{\real}{\ensuremath{{\mathbb R}}}

\newcommand{\PP}{{\mathcal P}}

\newcommand{\BB}{{\mathcal B}}

\newcommand{\EE}{{\mathcal E}}

\newcommand{\CC}{\mathcal{C}}

\newcommand{\LL}{\mathcal{L}}
\newcommand{\MM}{\mathcal{M}}
\newcommand{\NN}{\mathcal{N}}

\newcommand{\RR}{\mathcal{R}}
\newcommand{\OO}{\mathcal{O}}

\newcommand{\To}{\longrightarrow}

\newcommand{\gr}{\mathfrak}
\newcommand{\scr}{\mathscr}

\newcommand{\Vect}{\ensuremath{{\mathrm{Vect}}}}

\newcommand{\lf}{\ensuremath{\lfloor}}
\newcommand{\rf}{\ensuremath{\rfloor}}

\begin{document}

\title[T-duality, generalized geometry and dg-manifolds]{ T-duality and exceptional generalized geometry through symmetries of dg-manifolds.}

\author{Ernesto Lupercio, Camilo Rengifo and Bernardo Uribe}
\thanks{Part of the research
on this work was carried out while the third author enojoyed a ``Humboldt Research Fellowship for Experienced Researchers''.
The third author acknowledges and thanks the financial support of the Von Humboldt Foundation. }
\address{Departamento de Matem\'{a}ticas, CINVESTAV, 
Apartado Postal 14-740, C.P. 07000, M\'{e}xico, D.F., M\'{E}XICO. Tel:  +52 5550613868
Fax: +52 5550613876}
\email{lupercio@math.cinvestav.mx}
\address{Departamento de Matem\'{a}ticas, Universidad de los Andes, Carrera 1 N. 18A - 10, Bogot\'a, COLOMBIA. Tel: +571 3394999 ext. 2710 Fax: +571 3324427}
\email{cam-reng@uniandes.edu.co}
\address{Departamento de Matem\'{a}ticas y Estad\'{\i}stica, Universidad del Norte, Km 5 via Puerto Colombia, Barranquilla, Colombia}
\email{buribe@gmail.com}
\begin{abstract}
We study dg-manifolds which are  $\real[2]$-bundles over $\real[1]$-bundles over manifolds, we calculate
its symmetries, its derived symmetries and we introduce the concept of T-dual dg-manifolds.
 Within this framework, we construct the T-duality map
as a degree -1 map between the cohomologies of the T-dual dg-manifolds and we show
an explicit isomorphism between the differential graded algebra of the symmetries of the T-dual 
dg-manifolds. We, furthermore, show how the algebraic structure
underlying $B_n$ generalized geometry could be recovered as derived dg-Leibniz algebra of the fixed points of
the T-dual automorphism acting on the symmetries of a self T-dual dg-manifold, and we show how
other types of algebraic structures underlying exceptional generalized geometry could be obtained 
 as derived symmetries of certain dg-manifolds.
\end{abstract}
\maketitle
\section{Introduction}
Generalized complex manifolds and Dirac structures are geometric structures defined
in \cite{Hitchin, Gualtieri} and \cite{Courant} respectively, which among other things, provide a unified framework on which symplectic manifolds, Poisson manifolds,
complex manifold and certain types of foliations can be treated in a uniform way.

The  algebraic structures underlying these geometries are  the so called \emph{exact Courant algebroids}. These
algebroids are split extensions of the tangent bundle by the cotangent bundle, and are endowed with a Leibniz 
bracket known as the Courant-Dorfman bracket. Dorfman \cite{Dorfman}, in its original construction, obtained
the Courant-Dorfman bracket by applying the derived bracket procedure to a 4-tiered differential graded
Lie algebra associated to vector fields, functions, 1-forms and 2-forms. It was noted by {\v{S}}evera that 
the construction of Dorfman could be recovered and generalized to the twisted case by performing the derived
bracket procedure to the differential graded Lie algebra of symmetries of an $\real[2]$-bundle over $T[1]M$ in the language
of dg-manifolds (cf. \cite{Severa-homotopy}), so it was in this framework of dg-manifolds that the third author \cite{Uribe} explored
the conditions under which a Lie group action on a manifold could be lifted to an action on an exact Courant algebroid  

In recent works on string theory \cite{Hull, PachecoWaldram} it has been proposed to consider a geometric structure called
\emph{exceptional generalized geometry}, an structure similar in nature
to the geometry of generalized complex manifolds in the sense that it incorporates vector fields and differential forms of higher degree in order to study backgrounds on eleven-dimensional supergravity. The underlying algebraic structure of this exceptional generalized geometry was further studied in \cite{Baraglia} where it is described how this structure can be obtained from 
a certain simple Lie algebra and, at the same time, there are also described other related types of algebraic structures coming from other
simple Lie algebras.
Motivated by the relation developed by {\v{S}}evera between Exact Courant algebroids and the derived symmetries of $\real[2]$-bundles
over $T[1]M$, in this work, we investigate how to obtain the algebraic structures, defined in \cite{Baraglia}, as 
derived constructions from the symmetries of particular dg-manifolds. 

One of the exceptional generalized geometries introduced in \cite{Baraglia} under the name of
\emph{$B_n$ generalized geometry} lives on the bundle $TM \oplus \real \oplus T^*M$ and
can be twisted by a pair of a closed 2-form $F$ and a 3-form $H$ satisfying $dH + F \wedge F=0$. Now, since 
the structural constants which define a $\real[2]$-bundle over a $\real[1]$-bundle over $T[1]M$ are  closed
2-forms $F$ and $\bar{F}$ and a 3-form $H$ satisfying $dH + F \wedge \bar{F}=0$, we note that the ``$B_n$ generalized
geometry" should be related to the symmetries of ``self dual bundles", namely the ones on which $F = \bar{F}$; 
this basic observation is the basis of this investigation.

Let us describe what we accomplish in this work. We start by reviewing the basics on dg-manifolds and on $\real[n]$-bundles over
dg-manifolds. Then we study $\real[2]$-bundles over $\real[1]$-bundles over $T[1]M$ and we find that the structural constants
of these bundles are closed 2-forms $F$ and $\bar{F}$ and a 3-form $H$ satisfying $dH + F \wedge \bar{F}=0$ making
the vector field $d + F \partial_q + (H + q\bar{F})\partial_t$ homological. Then, we note that, since the equations on $F$ and $\bar{F}$ are
symmetric, we see that the homological vector field $d + \bar{F} \partial_q + (H + q{F})\partial_t$ defines
another $\real[2]$-bundle over another $\real[1]$-bundle and we remark that this construction is completely analogous to the
original construction of topological T-duality for circle bundles; hence, it is natural to call the two resulting dg-manifolds T-dual. 
We follow the construction of the T-duality isomorphism in twisted cohomologies done in \cite{BouwknegtMathai1}  and we put
it in the framework of dg-manifolds; we obtain a degree -1 map between the complexes that calculate the cohomology
 of the T-dual dg-manifolds, and we show that this map fits on the right hand side of a short exact sequence of complexes where the
 left hand side is given by the inclusion of the De Rham complex of $M$ into the complex of the first dg-manifold. This degree -1 map,
 in cohomology of degrees greater than the cohomology of $M$, produces an isomorphism and in particular it induces the known
 isomorphism in twisted cohomology for the T-dual pair, this is Theorem \ref{theorem T-duality map}.

 Then, we proceed to explicitly calculate the symmetries and their derived symmetries of $\real[2]$-bundles over $\real[1]$-bundles
 and we show that the symmetries and derived symmetries of T-dual pairs are isomorphic through a simple map
 that switches tangent information with cotangent information; this is Theorem \ref{theorem T-duality iso for derived symmetries}.
 We note, in Corollary \ref{corollary T-duality for Courant algebroids}, that this isomorphism provides an alternative proof for the T-duality isomorphism al the level of the Courant algebroids for dual
 pairs carried out in \cite{CavalcantiGualtieri}.
 
Then, we find that the algebraic structure underlying  $B_n$ generalized geometry in the sense of Baraglia \cite{Baraglia} can be obtained from a self T-dual
 dg-manifold by looking at the derived symmetries fixed by the T-duality automorphism; this is Theorem \ref{theorem Bn geometry}. The latter, 
 in particular, implies that the $B_n$ generalized geometry embeds into the generalized geometry of the principal $\real[1]$-bundle defined by
 the curvature form $F$. We finish by giving an alternative description of the underlying algebraic structure of $E_6$ generalized
 geometry as the derived symmetries of a $\real[6]$-bundle over a $\real[3]$-bundle over $T[1]M$.
 
 The results of this paper seem to support the general intuition (observed by many people, including Nigel Hitchin) that generalized geometry seems to be a way to encode a sort of poor man's super-geometry, as dg-manifolds are an enriched form of super-manifolds, and they dovetail nicely with the previous work of Schwarz \cite{schwarz1996superanalogs}. Judging by the relevance of the differential-graded viewpoint to exact Courant algebroids, one would expect the results in this paper to bring useful insights into exceptional generalized geometry.
 
 A motivation for this last assertion comes from the geometric realm of derived differential geometry and higher categories; for dg-manifolds are first order approximations to derived stacks. Indeed, one of the most influential papers about global deformation theory (in the context of topological quantum field theories) is Kontsevich's paper on enumeration of rational curves \cite{KontsevichTorus95}, used there to define and compute the virtual fundamental class of the moduli space of stable curves. In order to do this Kontsevich uses the concept of \emph{quasi-manifold}, that is to say,  space with charts $U_i$, that locally looks like the intersection of two manifolds inside a larger one: $U_i \cong Y_i \cap Z_i \subseteq V_i.$ Unfortunately the descent theory required to make this work in a rigorous manner seems formidable: the natural way to control the gluing is through the theory of higher stacks \cite{simpson1997closed}. A satisfactory foundational framework for this avenue of thought has been developed by Lurie \cite{Lurie10} and by T\"oen \cite{Toen09, toen2014derived}. This approach has already been fruitful in mathematical physics \cite{PT13}.
 
 The approach of this paper is different. Alternatively in \cite{CFK01, CFK02, CFK09} Ciocan-Fontanine and Kapranov have introduced the notion of \emph{dg-manifold}, in order to formalise Kontsevich's notion of virtual fundamental class. The notion of dg-manifold while only an approximation to the notion of derived higher stack is, nonetheless, enough for this task. It provides a rigorous and simple framework to work in mathematical-physics without the large foundational weight of the theory of derived higher stacks. This is the approach that we choose for this work to understand $T$-duality, and we are confident that it will provide useful in mathematical physics and string theory. 
 
 We would like to thank the referee for many useful comments that we believe have improved this work.
  
\section{$\real[n]$-bundles on dg-manifolds} 
\subsection{dg-Manifolds} Let us start with some notational conventions. Let $\NN$ be a differentiable
 (super)manifold and  let us denote its sheaf of smooth functions by
$\OO_\NN$. For $P=\{P_k\}_{k \in \integer}$
 a graded vector bundle over $\NN$, $S(P)$ will denote the the sheaf of graded
 commutative $\OO_\NN$-algebras  freely generated by $P$; the locally ringed space
 $(\NN,S(P^*))$ will also be denoted by $P$ where $P^*$ is the dual vector bundle.
 For an integer $k$, $P[k]$ denotes the shifted vector bundle with $P[k]_l:=P_{k+l}$.
To keep the notation simple, we will usually denote a vector bundle and its $\OO$-module
 of sections with the same symbol.

\begin{definition}
A (non-negatively) {\it graded manifold}  is a locally ringed
space $\PP=(\NN, \OO_\PP)$, which is locally isomorphic to $(U,\OO_U
\otimes S(P^*))$, where $U \subset \real^{m|r}$ is an open domain
of $\NN$ and $P=\{P_i\}_{-n \leq i \leq 0}$ is a finite dimensional
negatively graded (super)vector space. The number $n$ is called
the degree of the graded manifold $P$.
\end{definition}
The global sections of $\PP$ will be called the functions on $\PP$ and
they will be denoted by $C^* (\PP)$, and the derivations of $C^*(\PP)$ will
be the vector fields of $\PP$ and they will be denoted $\Vect^*(\PP)$.
\begin{definition}
A {\it differential graded manifold} (dg-manifold) is a graded
manifold $\PP$ equipped with a degree 1 vector field $Q$ of $\Vect^1(\PP)$
satisfying $[Q,Q]/2 = Q^2=0$ (such a field $Q$ will be called a \emph{homological vector field}).
\end{definition}

Morphisms of dg-manifolds are morphisms of locally ringed spaces
respecting the homological vector field.  We recommend \cite{Voronov-graded, Roytenberg} for
an introduction to the theory of differential graded manifolds.

A dg-manifold over a point of degree $n$ is the same as an
$L^\infty$-algebra of degree $n$, also called Lie $n$-algebra. A
dg-manifold of degree $n$ is what is known as a ``Lie $n$
algebroid".

\subsubsection{Manifolds}
Given a manifold $M$, let us consider the graded manifold $T[1]M$, whose algebra of functions 
is the algebra of differential forms
$$C^*(T[1]M) = \Omega^\bullet M,$$
and endow it with homological vector field 
$$d= dx^i \frac{\partial}{\partial x_i},$$
the degree 1 derivation which is the De Rham differential. Let us denote by 
$$\MM=(T[1]M, d)$$
the image of $M$ in the category of dg-manifolds

\subsection{$\real[n]$-bundles} A $\real[n]$-bundle over the dg-manifold $\NN= ( \OO_\NN, Q_\NN)$ consists of the graded
manifold $\OO_\NN \oplus \real[n]$ together with a choice of derivation of degree 1
$$Q_\NN + \Theta \partial_t$$
of the algebra of functions
$$C^*(\OO_\NN) \otimes S[t]$$
that satisfies the Maurer-Cartan equation, i.e. $Q_\NN \Theta=0$.
Here $\Theta$ is a function of degree $n+1$ of $\NN$ and $S[t]$ denotes the symmetric graded algebra generated by the parameter $t$
whose degree is $n$.

\subsubsection{$\real[1]$-bundles over $T[1]M$} A $\real[1]$-bundle over the dg-manifold $\MM$ is given by the graded manifold $T[1]M \oplus \real[1]$ together
with a choice of homological vector field that lifts $d$.  The algebra of functions of $T[1]M \oplus \real[1]$ becomes 
$$C^*(T[1]M \oplus \real[1])=\Omega^\bullet M \otimes \Lambda [q]$$ where $q$ is a variable of degree 1. A 
generic choice of derivation of degree 1 which lifts $d$ is given by
$$d + F \partial_q$$ 
where $F$ is a 2-form on $M$, and this derivation becomes homological
if it satisfies the Maurer-Cartan equation
$$[d + F \partial_q,d + F \partial_q]=0;$$
which is the same as saying that $dF=0$. So, any choice of closed 2 -form will define a $\real[1]$-bundle over $\MM$.
Let us denote this dg-manifold by
$$\RR:= (T[1]M \oplus \real[1], d + F \partial_q).$$ 

\subsubsection{Relation to principal $S^1$-bundles} Let $S^1 \to R \stackrel{\pi}{\to} M$ be a principal $S^1$ bundle and
denote by $\theta \in \Omega^1R$ any connection 1-form and by $F \in \Omega^2 M$ the associated curvature form; we have that $d_R \theta = \pi^*F$,
where $d_R$ denotes the De Rham differential in $R$. 
Consider also the dg-manifold $\RR:= (T[1]M \oplus \real[1], d + F \partial_q)$ where $F$ is the curvature 2-form of the $S^1$-bundle

We claim that
\begin{proposition} \label{proposition isomorphism with circle bundles}
Consider the principal bundle $S^1 \to R \stackrel{\pi}{\to} M$. Then
the map of graded algebras
\begin{align*}
j : \Omega^\bullet M \otimes \Lambda [q] & \to \Omega^\bullet R\\
\omega_0 + q \omega_1 & \mapsto \pi^*  \omega_0 + \theta   \pi^* \omega_1
\end{align*}
 induces a quasi-isomorphism
 \begin{align*}
j : (\Omega^\bullet M \otimes \Lambda [q], d_M + F \partial_q) & \to (\Omega^\bullet R, d_R)
\end{align*}
 \end{proposition}
 \begin{proof}
 Let us prove first that the map $j$ is a morphism of complexes; on the left hand side we have that:
 \begin{align*}
d_R(j(\omega_0 + q \omega_1 )) &= d_R(\pi^*  \omega_0 + \theta   \pi^* \omega_1) \\
&= d_R(\pi^*  \omega_0) + \pi^*F   \pi^* \omega_1- \theta d_R ( \pi^* \omega_1)\\
&= \pi^* (d_M \omega_0) + \pi^*F   \pi^* \omega_1- \theta  \pi^* (d_M \omega_1),
\end{align*}
 while on the right hand side, we can write:
 \begin{align*}
j( d_M + F \partial_q(\omega_0 + q \omega_1 )) &= j(d_M  \omega_0 + F   \omega_1 - q d_M \omega_1) \\
&= \pi^* (d_M \omega_0) + \pi^*F   \pi^* \omega_1- \theta  \pi^* (d_M \omega_1),
\end{align*}
hence $d_R \circ j = j \circ ( d_M +F \partial_q)$.

Let's consider the filtration associated to a good cover on $M$.  Such filtration induces a Serre spectral sequence of the fibration $S^1 \to R \to M$
 which converges to $H^*(R)$; let us denote this spectral sequence by $\overline{E}^{p,q}_r$. Let's consider a corresponding filtration
associated to $\Omega^\bullet M \otimes \Lambda [q]$ via the same open sets, which converges to 
$H^*(\Omega^\bullet M \otimes \Lambda [q], d_M + F \partial_q)$; let´s denote this spectral sequence by ${E}^{p,q}_r$. 
Since the filtrations are obtained via the same
open sets, we have that $j$ induces a map of spectral sequences
$$j: {E}^{p,q}_r \to \overline{E}^{p,q}_r.$$

The second page of $E$ is equal to 
$$E_2^{p,q} = H^p(M) \otimes \Lambda[q]$$
and since the fundamental group of $M$ acts trivially on $H^*(S^1)$, we have 
that
$$\overline{E}^{p,q}_2 = H^p(M) \otimes H^q(S^1).$$
The map $j$ induces an isomorphism of the second pages of the spectral sequences, and, therefore, it induces an isomorphism
of the infinity pages. Since the complexes are bounded, the map $j$ induces an isomorphism of the cohomologies
$$j : H^* ( \Omega^\bullet M \otimes \Lambda [q], d_M + F \partial_q) \stackrel{\cong}{\to} H^*(R).$$
 \end{proof}

\subsubsection{$\real[n]$-bundles on $T[1]M$} A $\real[n]$-bundle over $T[1]M$ is determined by a closed $n+1$ form $\Theta \in \Omega_{\rm{cl}}^{n+1}M$ making the derivation
$$d + \Theta \partial_t$$
of the algebra of functions $\Omega^\bullet M \otimes S[t]$
homological.

In view of the result of Proposition \ref{proposition isomorphism with circle bundles},  we should think of the dg-manifold
$$( T[1]M \oplus \real[n], d + \Theta \partial_t)$$
as the analogue of a principal $K(\integer, n)$-bundle $K(\integer, n) \to P \to M$ over the manifold $M$. Since we do not have a 
reasonable model for the differential forms for $K(\integer,n)$ for $n \geq 2$, we cannot proceed as in the previous proposition.
Instead, we could think of the dg-manifold $( T[1]M \oplus \real[n], d + \Theta \partial_t)$ as a rigid model for the differential forms on $P$,
provided that the differential form $\Theta$ has integral periods.
 In particular, we should think that the cohomology
 $$H^*(\Omega^\bullet M \otimes S[t],d + \Theta \partial_t)$$
 calculates the cohomology of the total space $P$.

\subsubsection{$\real[2]$-bundles on $T[1]M$ and twisted cohomology} A $\real[2]$-bundle over $T[1]M$ is the dg-manifold
$$\PP = (T[1]M \oplus \real[2], Q_\PP=d + H \partial_t)$$
for $H$ a closed 3-form on $M$.  The cohomology of $\PP$ is the cohomology of $\Omega^\bullet M \otimes \real[t]$ with respect
to the differential $d + H \partial_t$ where $t$ is a variable of degree 2. Denote this cohomology by
$$H^*(\PP) = H^* (\Omega^\bullet M \otimes \real[t], d + H \partial_t).$$

A fiber-wise gauge transformation is  given by a the degree 0 vector field $B \partial_t$ where $B$ is any 2-form on $M$,
transforming the homological vector field by
$$d + H \partial_t \mapsto (d + H \partial_t + [ d + H \partial_t, B \partial_t] )= d + (H + dB) \partial_t.$$
The map 
$$e^{B \partial_t} : \Omega^\bullet M \otimes \real[t] \to \Omega^\bullet M \otimes \real[t]$$
induces an isomorphism between the dg-manifolds
$$e^{B \partial_t} : \PP' \stackrel{\cong}{\to} \PP$$ 
where $\PP'= (T[1]M \oplus \real[2], d + (H+dB) \partial_t)$ and, therefore, it induces an isomorphism in cohomologies
$$e^{B \partial_t}:H^*(\PP') \stackrel{\cong}{\to} H^*(\PP).$$

For a closed 3-form $H$,  the $H$-twisted cohomology of $M$ is defined as the cohomology of the $\integer /2$-graded
 complex $\Omega^{\rm{ev,od}} M$, with respect to the differential $d + H \wedge$, and this cohomology is defined by
 $$H^{{ev,od}}(M; H) = H^*(\Omega^{{ev,od}} M, d + H \wedge).$$
 
 Similarly, the map $e^{B \wedge}: \Omega^{{ev,od}} M \to \Omega^{{ev,od}} M$ produces an isomorphism of twisted complexes
 $$e^{B \wedge}:( \Omega^{{ev,od}} M , d+ (H+dB)\wedge ) \stackrel{\cong}{\to} ( \Omega^{{ev,od}} M , d+ H\wedge )$$
 inducing an isomorphism
 $$H^{{ev,od}}(M; H+dB) \cong H^{{ev,od}}(M; H).$$
 
We claim

\begin{proposition} \label{proposition iso twisted non-twisted}
Let $m$ be the dimension of $M$ and consider a closed 3-form $H$ on $M$.
Let  $$\PP = (T[1]M \oplus \real[2], d + H \partial_t)$$ be the associated dg-manifold
and let 
$$H^{{ev,od}}(M; H) = H^*(\Omega^{{ev,od}} M, d + H \wedge)$$
be the twisted cohomology of $M$ twisted by $H$.
Then, there are isomorphisms of groups
$$H^{ev}(M; H) = H^{2k}(\PP), \ \ H^{od}(M; H) = H^{2k+1}(\PP)$$
for any $k$ provided that the degree $2k$ or $2k+1$ respectively is greater than $m$.
\end{proposition}
 \begin{proof}
 
 Consider the homomorphism of vector spaces given by the rescaling
 \begin{align*}
 \phi: \Omega^\bullet M \otimes \real[t] & \to \Omega^{{ev,od}} M\\
  \omega t^k & \mapsto  k!  \cdot \omega,
 \end{align*}
 and note that $\phi$ preserves the $\integer /2$ grading.
 
 Since $$\phi((d + H \partial_t)\omega t^k)  = \phi(d \omega t^k + k H \omega t^{k-1})=  k! (d\omega  +  H \omega )$$
 we see that
 $$((d + H \wedge) \phi(\omega t^k)) = \phi((d + H \partial_t)\omega t^k),$$
 provided that $k \geq 1$.  
 
 If we denote by $C^j(\PP)$ the homogeneous forms on $\Omega^\bullet M \otimes \real[t]$ of degree $j$, we see that
 the map $\phi$ induces a map of complexes 
 $$\phi:  (C^{* >d}(\PP), d + H \partial_t)   \to (\Omega^{{ev,od}} M, d + H \wedge).$$
and moreover it induces an isomorphism of vector spaces
 $$\phi: C^{2k}(\PP) \stackrel{\cong}{\to} \Omega^{ev}M \ \ \ C^{2k+1}(\PP) \stackrel{\cong}{\to} \Omega^{od}M$$
 provided that $2k \geq m$.
 
 We conclude that the induced maps on cohomologies
 $$\phi : H^{2k}(\PP) \stackrel{\cong}{\to} H^{ev}(M,H), \ \ \phi : H^{2k+1}(\PP) \stackrel{\cong}{\to} H^{od}(M,H)$$
 are isomorphisms whenever $2k>m$ for the left hand side and $2k+1>m$ for the right hand side. 
  \end{proof}
 
 As a corollary, we have that the cohomology of $\PP$ becomes periodic for degrees greater than the dimension of the manifold $M$;
 that is
 $$H^j(\PP) \cong H^{j+2l}(\PP) \ \ \ \mbox{for} \ ÊÊ\  l\geq 0 \ Ê\ \mbox{and} \  \ j > \mbox{dim} (M).$$ 
 
 \begin{rem}
 The $H$-twisted cohomology $$H^*(M,H)$$ could be understood as the cohomology of the dg-manifold 
 $$\PP = (T[1]M \oplus \real[2], d + H \partial_t)$$
 for degrees greater than the dimension of the manifold $M$. The previous simple remark explains
 the following somewhat puzzling fact: one sometimes encounters some twisted cohomologies are zero in even \emph{and} in odd degrees.
  This  is explained  by noting that the cohomology
 of the associated dg-manifold $\PP$ is zero for degrees greater than the degree of the manifold, nevertheless this does not say anything
 about the cohomology of $\PP$ in lower degrees.
 
For example, when $M = S^3$ and $H$ is a volume form for the sphere, we have that $H^*(S^3, H)=0$ in both degrees. This
means that  the cohomology of $\PP=(T[1]S^3 \oplus \real[2], d + H \partial_t)$ is zero for degrees greater than 3; and this can be 
easily checked since the cohomology of $\PP$ is
$$H^0(\PP)=\real \ \  Ê{\rm{and}} \ \  H^{*>0}(\PP)=0.$$

With this setup in mind, most of the properties that  twisted cohomology satisfies are easily derived from the 
properties that ordinary cohomology satisfies. Hence, the twisted cohomology becomes a $\integer /2$-graded cohomology
theory satisfying the axioms of a generalized cohomology theory.
 \end{rem}
 
\subsubsection{Integration on the fibers on $\real[n]$-bundles for $n$ odd} \label{section pushforward}

Let $n$ an odd positive integer and consider the $\real[n]$-bundle $$ \RR=(\NN \oplus \real[n] , Q_\RR=Q_\NN + \Theta \partial_q)$$
over the dg-manifold $(\NN, Q_\NN)$ where $q$ is an odd variable of degree $n$, and $\Theta$ is a function of $\NN$ of degree $n+1$
such that $Q_\NN \Theta=0$;  denote the bundle map $ p: \RR \to \NN$.

Let $p_*$ be the pushforward map;
\begin{align*}
p:C^*(\NN) \otimes \Lambda[q] & \to C^*(\NN) \\
\omega q & \mapsto \omega\\
\eta & \mapsto 0,
\end{align*}
where $\omega$ and $\eta$ are functions on $\NN$. Note that it is a degree $-n$ map from $C^*(\RR)$ to $C^*(\NN)$ and, moreover, that $p_*$ is a map of complexes, i.e.
$$p_* \circ Q_\RR = Q_\NN \circ p_*.$$
Therefore, we have that the pushforward 
$$p_* : (\NN \oplus \real[n] , Q_\RR=Q_\NN + \Theta \partial_q) \to (\NN , Q_\NN)$$
is a map of dg-manifolds of degree $-n$ and, therefore, it induces a map of degree $-n$ on the cohomologies
$$p_*: H^*(\RR,Q_\NN + \Theta \partial_q) \to H^{*-n}(\NN, Q_\NN)$$
associated to the dg-manifolds.

\section{$\real[2]$-bundles over $\real[1]$-bundles over $T[1]M$ and T-duality} \label{section R2-R1-bundles}
Let $\RR=(T[1]M \oplus \real[1], Q_\RR=d + F \partial_q)$ be a $\real[1]$-bundle over $T[1]M$. 

A $\real[2]$-bundle over the dg-manifold $\RR$ is given by the graded manifold $(T[1]M \oplus \real[1])\oplus \real[2]$ together
with a choice of homological vector field that lifts $d+ F \partial_q$.
The algebra of functions of $(T[1]M \oplus \real[1])\oplus \real[2]$ becomes 
$$C^*((T[1]M \oplus \real[1])\oplus \real[2])=\Omega^\bullet M \otimes \Lambda [q] \otimes \real [t]$$ where $t$ is a variable of degree 2. A 
generic choice of derivation of degree 1 which lifts $d + F \partial_q$ is given by
$$d + F \partial_q +(H +  q\bar{F} ) \partial_t$$ 
where $H$ is a 3-form on $M$ and $\bar{F}$ is a 2-form on $M$; this derivation becomes homological
if it satisfies the Maurer-Cartan equation
$$[d + F \partial_q +(H + q\bar{F} ) \partial_t,d + F \partial_q+(H + q\bar{F} ) \partial_t]=0.$$
This last equation is equivalent to the equations
\begin{align*} dF&=0 \\d\bar{F} &=0\\dH+ F\wedge \bar{F}&=0.\end{align*}
So, a closed 2 -form $\bar{F}$ on $M$ such that the cohomology class 
$[F \wedge \bar{F}]=0$ vanishes, together with a choice of $3$ form $H$ which makes $F \wedge \bar{F}$ exact i.e. 
$$dH + F\wedge \bar{F}=0$$ 
will define a $\real[2]$-bundle over $\RR$.
Let us denote this dg-manifold by
$$\PP:= ((T[1]M \oplus \real[1]) \oplus \real[2], d + F \partial_q +(H + q\bar{F}) \partial_t ).$$ 

\begin{rem}
Note that the 3-form $H+q\bar{F}$ is automatically closed under the differential $d +F\partial_q$ since this differential is the one that
defines the structure on $\RR$:
$$(d +F\partial_q )(H+q\bar{F} )= dH +F\wedge \bar{F} + qd\bar{F}  =0.$$
\end{rem}

\subsection{Gauge transformations} In order to find when two homological vector fields 
on $(T[1]M \oplus \real[1]) \oplus \real[2]$
are homologous, we calculate the bracket
$$[d + F \partial_q +(H +  q\bar{F} )\partial_t,   \LL_X +A \partial_q +(B +  q\bar{FA} )\partial_t]$$
for $A$ and $\bar{A}$ 1-forms, $B$ a 2-form and $\LL_X$ the Lie derivative of a vector field $X$ of $M$; the element
$$\LL_X +A \partial_q +(B +  q\bar{A} )\partial_t$$ is a generic
 derivation of degree 0 of $\Omega^\bullet M \otimes \Lambda [q] \otimes \real [t]$.
The bracket becomes 
\begin{align*}
[d + F \partial_q +(H +  q\bar{F} )\partial_t,   \LL_X +A & \partial_q +(B + q \bar{A} )\partial_t]=\\
&(dA -\LL_X F )\partial_q  \\
+&(dB -\LL_X H +  F \wedge \bar{A} - A \wedge \bar{F})\partial_t\\
-&q ( d\bar{A}+ \LL_X \bar{F}) \partial_t \end{align*}
which implies that the homological vector fields
$$d + F \partial_q +(H +  q\bar{F})\partial_t$$
and
$$ d + (F +dA -\LL_X F )\partial_q + (H+dB -\LL_X H + F \wedge \bar{A}-  A \wedge \bar{F}+  q(\bar{F}- d\bar{A}-\LL_X \bar{F}) ) \partial_t$$
are homologous. 

\subsubsection{Fiber-wise gauge transformations} In the particular case on which $X=0$, we get that the gauge transformation, defined by 
$A,\bar{A}$ and $B$; via the degree 0 derivation $A \partial_q +(B +  q\bar{A})\partial_t$, transforms $F,\bar{F}$ and $H$ in the following way:
\begin{align*}
F & \mapsto F +dA\\
H & \mapsto H + dB  + F \wedge \bar{A} -  A \wedge \bar{F}\\
\bar{F}& \mapsto \bar{F} -d\bar{A}.
\end{align*}

\subsection{T-duality} \label{subsection T-duality} In this section, we will write the T-duality construction in the language of dg-manifolds.
We recover the isomorphism of twisted cohomologies of the dual pairs by producing a degree -1
morphism between the cohomologies of the dual dg-manifolds.
 We believe that this alternative point of view might be of interest.
 
 \subsubsection{} Consider the $\real[1]$-bundle over $\MM=(T[1]M, d)$
 $$\EE=(T[1]M \oplus \real[1], Q_\EE=d + F\partial_q)$$
 where $F$ is a closed 2-form on $M$ and $q$ is a variable of degree 1.
 Recall from Proposition \ref{proposition isomorphism with circle bundles} that if $F$ is the curvature
form of a circle bundle $E \to M$, then there is a canonical isomorphism $H^*(\EE)\cong H^*(E)$. 

Consider the $\real[2]$-bundle over $\EE$
$$\PP=((T[1]M \oplus \real[1])\oplus \real[2], Q_\PP=d + F\partial_q + (H + q \bar{F})\partial_t)$$
where $H + q \bar{F}$ is a closed three form on $\EE$; hence we have that
$\bar{F}$ is closed and $dH + F \bar{F}=0$. 

Since the equations $dF=0=d\bar{F}$ and  $dH + F \bar{F}=0$ are symmetric on $F$ and $\bar{F}$, we can consider the $\real[1]$-bundle
$$\bar{\EE}=(T[1]M \oplus \real[1], Q_{\bar{\EE}}=d + \bar{F}\partial_{\bar{q}})$$
together with the $\real[2]$-bundle over it 
$$\bar{\PP}=((T[1]M \oplus \real[1])\oplus \real[2], Q_{\bar{\PP}}=d + \bar{F}\partial_{\bar{q}} + (H + \bar{q} F)\partial_t).$$

Denoted by $\EE \times_\MM \bar{\EE}$ the dg-manifold 
$$\EE \times_\MM \bar{\EE}= ( T[1]M \oplus \real[1] \oplus \real[1],Q_{\EE \times_\MM \bar{\EE}}= d +  F\partial_{q}+\bar{F}\partial_{\bar{q}}),$$
note that this dg-manifold fits in the pullback of the diagram:
$$\xymatrix{ & \ar[ld]_{p} \EE \times_\MM \bar{\EE} \ar[rd]^{\bar{p}} & \\
\EE  \ar[rd]^\pi& & \ar[ld]_{\bar{\pi}}  \bar{\EE}\\
& \MM &
}$$
and consider the pullbacks of the $\real[2]$-bundles $\PP, \bar{\PP}$ to $\EE \times_\MM \bar{\EE}$ which becomes the dg-manifolds
\begin{align*}
p^*\PP &= ((T[1M \oplus \real[1] \oplus \real[1])\oplus \real[2],Q_{p^*\PP}=
 d +  F\partial_{q}+\bar{F}\partial_{\bar{q}} + (H + q \bar{F})\partial_t)\\
 \bar{p}^*\bar{\PP} &= ((T[1M \oplus \real[1] \oplus \real[1])\oplus \real[2],Q_{\bar{p}^*\bar{\PP}}=
 d +  F\partial_{q}+\bar{F}\partial_{\bar{q}} + (H + \bar{q} {F})\partial_t).
\end{align*}
Since the 3-forms $(H + q \bar{F})$ and $(H + \bar{q} {F})$ are cohomologous in $\EE \times_\MM \bar{\EE}$ via the 
2-form $\bar{q} q$, i.e.
$$(H + q \bar{F}) -(d +  F\partial_{q}+\bar{F}\partial_{\bar{q}})(\bar{q} q)=(H + \bar{q} {F}),$$
we have that the degree 0 map
$$e^{\bar{q} q \partial_t} : \Omega^\bullet M \otimes \Lambda[q, \bar{q}] \otimes \real[t] \stackrel{\cong}{\to}
 \Omega^\bullet M \otimes \Lambda[q, \bar{q}] \otimes \real[t]$$
 induces an isomorphism of dg-manifolds
 $$e^{\bar{q} q \partial_t} : p^*\PP \stackrel{\cong}{\to} \bar{p}^*\bar{\PP}.$$
Schematically we have 
$$\xymatrix{ &p^*\PP \ar[ld]^{\bf{p}}\ar[rd] \ar@{.>}@/^/[rr]^{e^{\bar{q} q \partial_t}}& &\bar{p}^*\bar{\PP}\ar[ld] \ar[rd]^{\bar{\bf{p}}}& \\  \PP \ar[rd] && \ar[ld]_{p} \EE \times_\MM \bar{\EE} \ar[rd]^{\bar{p}}& & 
\bar{\PP} \ar[ld]\\
& \EE  \ar[rd]^\pi& & \ar[ld]_{\bar{\pi}}  \bar{\EE}&\\
& &\MM &&
}$$
where $\bf{p}$ and $\bar{\bf{p}}$ denote the corresponding maps of dg-manifolds, and the dotted arrow corresponds to the vertical
gauge transformation defined previously.

Consider the  map of complexes \`a la Fourier-Mukai, which is the composition of the maps
$$ C^*( \PP) \stackrel{\bf{p}^*}{\To} C^*(p^*\PP) \stackrel{e^{\bar{q}q\partial_t}}{\To} C^*(\bar{p}^*\bar{\PP})
 \stackrel{\bar{\bf{p}}_*}{\To} C^*(\bar{\PP})[1]$$
where the pushfoward map $\bar{\bf{p}}_*$ is the one that was defined in section \ref{section pushforward}, and
$C^*(\bar{\PP})[1]^k= C^{k-1}(\bar{\PP})$ . Denoting the composition by
$$T:= \bar{\bf{p}}_* \circ  e^{\bar{q}q\partial_t}\circ \bf{p}^*$$
we get that the map
$$T: ( C^*( \PP), Q_\PP) \To (C^*(\bar{\PP})[1], Q_{\bar{\PP}})$$
is a map of complexes.

\begin{definition}
The pair of dg-manifolds $( \EE, \PP)$ and $(\bar{\EE},\bar{\PP})$ over $\MM$ are said to be {\it{T-dual}} and they satisfy the usual
relations\begin{align*}
\pi_*(H + q \bar{F}) &=\bar{F}\\
\bar{\pi}_*(H + \bar{q} {F}) &={F}.
\end{align*}
The map of complexes of degree -1
$$T: ( C^*( \PP), Q_\PP) \To (C^*(\bar{\PP})[1], Q_{\bar{\PP}})$$
is called the {\it{T-duality map}}.
\end{definition}

\begin{theorem} 
\label{theorem T-duality map}
Let $( \EE,\PP)$ and $(\bar{\EE}, \bar{\PP})$ be a T-dual pair of $\real[2]$-bundles over $\real[1]$-bundles over $M$.
Then the T-dual map 
$$T: ( C^*( \PP), Q_\PP) \To (C^*(\bar{\PP})[1], Q_{\bar{\PP}})$$
fits in the short exact sequence of complexes
\begin{align*}
0 \To (\Omega^* M, d) \hookrightarrow ( C^*( \PP), Q_\PP) \stackrel{T}{\To} (C^*(\bar{\PP})[1], Q_{\bar{\PP}}) \To 0
\end{align*}
and, therefore, it induces a long exact sequence in cohomologies
\begin{align*}
\to H^{k}(\PP, Q_\PP)  \stackrel{T}{\to} H^{k-1}(\bar{\PP}, Q_{\bar{\PP}}) \to H^{k+1}M 
\to H^{k+1}(\PP, Q_\PP)  \stackrel{T}{\to} H^{k}(\bar{\PP}, Q_{\bar{\PP}})\to
\end{align*}
In particular, when $k \geq{\rm{dim}} (M)$ the T-dual map induces an isomorphism in cohomologies
$$T:H^{k+1}(\PP, Q_\PP)  \stackrel{\cong}{\to} H^{k}(\bar{\PP},Q_{\bar{\PP}}).$$
\end{theorem}

\begin{proof}
Let us calculate explicitly the map $T$ for generic elements of the form $\omega t^j$ and $q\eta  t^l$ for $\omega$ and $\eta$
differential forms on $M$. We have that
\begin{align*}
T(\omega t^j) &= \bar{\bf{p}}_* (e^{\bar{q}q\partial_t}(\omega t^j))= \bar{\bf{p}}_* (\omega t^j + j\bar{q}q \omega t^{j-1})=(-1)^{|\omega|} j\bar{q} \omega t^{j-1}\\
T(q \eta  t^l) &= \bar{\bf{p}}_* (e^{\bar{q}q\partial_t}(q\eta  t^l))= \bar{\bf{p}}_* (q\eta  t^l) =(-1)^{|\eta|}  \eta t^{l}
\end{align*}
and, therefore, $T(\omega)=0$ for $\omega$ a differential form on $M$. Therefore
$${\rm{ker}}(T) = \Omega^*M$$
and since $Q_\PP$ restricted to $\Omega^*M$ is precisely the De Rham derivation, we have the desired
short exact sequence of complexes
\begin{align*}
0 \To (\Omega^* M, d) \hookrightarrow ( C^*( \PP), Q_\PP) \stackrel{T}{\To} (C^*(\bar{\PP})[1], Q_{\bar{\PP}}) \To 0.
\end{align*}

The connection homomorphism
$$\beta : H^{k-1}(\bar{\PP}, Q_{\bar{\PP}}) \to H^{k+1}M$$
is defined as follows: take a cocycle representing the cohomology class in $H^{k-1}(\bar{\PP}, Q_{\bar{\PP}})$ and
write it as $$\sum_{i \geq 0} (\omega_i + q \eta_i) t^i.$$
Then 
$$\beta(\sum_{i \geq 0} (\omega_i + q \eta_i) t^i):= (-1)^{|\omega_0|} F \omega_0$$
since we have that $T((-1)^{|\omega_0|}q \omega_0) = \omega_0$ and the restriction of 
$$(d + F \partial_q +(H + q \bar{F})\partial_t) ((-1)^{|\omega_0|}q \omega_0)$$to $\Omega^*M$ is precisely
$(-1)^{|\omega_0|} F \omega_0$.

The connection homomorphism is trivial whenever $|\omega_0| + 2$ is greater than the dimension of the manifold. Therefore
we have that, for $k \geq{\rm{dim}} (M)$, the T-dual map induces an isomorphism in cohomologies
$$T:H^{k+1}(\PP, Q_\PP)  \stackrel{\cong}{\to} H^{k}(\bar{\PP},Q_{\bar{\PP}}).$$
\end{proof}
Applying Proposition \ref{proposition iso twisted non-twisted} to the dg-manifolds $\EE$ and $\bar{\EE}$, we obtain the T-duality isomorphism
in twisted cohomologies that was proven in \cite{BouwknegtMathai1,BouwknegtMathai2}.
\begin{cor}
Let $( \EE,\PP)$ and $(\bar{\EE}, \bar{\PP})$ be a T-dual pair of $\real[2]$-bundles over $\real[1]$-bundles over $M$. Then the T-dual 
map induces an isomorphism of degree -1 of the twisted cohomologies
$$ H^{ev,od}(\EE, d + F \partial_q + (H + q \bar{F})\wedge) \stackrel{\cong}{\to} H^{od,ev}(\bar{\EE}, d + \bar{F} \partial_{\bar{q}} + (H +  \bar{q}F)\wedge)$$
\end{cor}

Here we should recall that Proposition \ref{proposition isomorphism with circle bundles} tells us that, in the case that $F$ is the curvature form of a $S^1$-principal bundle $E \stackrel{\pi}{\to} M $,
then the complex $C^*(\EE)$ is quasi-isomorphic to $\Omega^\bullet E$ and, therefore, the twisted complex
$$(C^*(\EE),d + F \partial_q + (H + q \bar{F})\wedge))$$
is quasi-isomorphic to the twisted complex
$$(\Omega^\bullet E, d_E + (H + \theta \bar{F})\wedge)$$
where $\theta$ is a choice of connection 1-form of the $S^1$-bundle with curvature $F$. This previous argument has as a corollary the usual T-dual isomorphism
for the twisted cohomology of T-dual pairs \cite{BouwknegtMathai1, BouwknegtMathai2}.

\subsection{Topological interpretation}  Recall that the rational homotopy type of the Eilenberg-Mclane spaces $K(\integer, n)$ is
equivalent to the graded symmetric algebras $ S[z]$ generated by one variable of degree $n$ i.e. $S[z]$ is a polynomial algebra
if $n$ is even and is an exterior algebra if $n$ is odd. 

Then we should think that the analogue in the category of topological spaces of a $\real[n]$-bundle over a dg-manifold is a
$K(\integer, n)$-bundle over a space. Recall that $K(\integer,n) = B^{n-1}S^1$ and that if $A$ is an abelian group, then $BA$ can be endowed with the structure of an abelian group. This means that the model $B^{n-1}S^1$ endows the Eilenberg-Maclane spaces with an
abelian group structure and, therefore, it is plausible to work with principal $K(\integer,n)$-bundles.

The analogue of the construction done in the previous section should be understood as a principal $K(\integer,2)$-bundle $P$ over 
the total space of a principal $S^1$-bundle $E$ over the manifold $M$.
The closed 2-form $F$ on $M$ should be thought as the curvature of the principal bundle $E$, 
the variable $q$ should be thought as the connection 1-form on $E$, and the 3-form $H + q\bar{F}$ on $E$
should be thought as the curvature of the gerbe that classifies the $K(\integer,2)$-bundle. Again, note that
the form $H + q\bar{F}$ is closed under the differential $d + F \partial_q$ if and only if $dH + F \bar{F} = 0$ and $d\bar{F}=0$.

\section{Symmetries of $\real[n]$-bundles}

\subsection{Symmetries of dg-manifolds} 
Recall that a homological vector field $Q$ on the graded manifold $P$ is the same as a Maurer-Cartan element on the
graded Lie algebra $\Vect^*(P)$.
Then any vector field   $\alpha \in \Vect^0(P)$ of degree 0 may define another Maurer-Cartan element by taking the action
on $Q$ of the exponential of the adjoint action of $\alpha$
$$Q \mapsto e^{({\rm ad}_\alpha)}Q:= Q + [\alpha, Q]  + \frac{1}{2} [\alpha, [\alpha,Q]] + \cdots$$
whenever we know that the series above converge. The infinitesimal version of this action is given by the adjoint action
of $\alpha$ on $Q$ and, therefore, the action is trivial whenever $[\alpha, Q]=0$.We say then that the infinitesimal symmetries of the Maurer-Cartan element are given by vector fields $\alpha$ of
degree 0 such that the adjoint action of $\alpha$ on $Q$ vanishes, i.e. $[\alpha, Q]=0$. Note that these infinitesimal symmetries
of $Q$ become a Lie algebra with respect to the brackets of $\Vect^0(P)$, as we have that for $\alpha_1$ and $\alpha_2$ commuting with $Q$,  the equality $[[\alpha_1,\alpha_2],Q]=0$ follows from the Jacobi identity and the fact that $Q$ is a homological vector field.

Furthermore, note that for any vector field $\beta$ of degree -1, the degree 0 vector field $[\beta,Q]$ commutes with $Q$ (again because of the Jacobi identity) and, therefore, it gives
an infinitesimal symmetry of $Q$. This means that we have to see the symmetries of the dg-manifold $P$ as a differential graded Lie algebra, where the differential is defined by the operator $[Q, \_]$ and the bracket is the one for vector fields.

\begin{definition}
Let $P$ be a dg-manifold with homological vector field $Q$. The (infinitesimal) symmetries of the
dg-manifold $P$ with homological vector field $Q$ is the differential
graded Lie algebra ${\gr{sym}}^*(P,Q)$ with
\begin{equation*}
{\gr{sym}}^q(P,Q) = \left\{
\begin{array}{lcl}
\Vect^q(P) & {\rm for} & q < 0 \\
\{ \alpha \in \Vect^0(P) | [\alpha,Q]=0 \} & {\rm for} & q =0 \\
0 & {\rm for} & q > 0 \\
\end{array}
\right.
\end{equation*}
 whose differential is $[Q,\_]$ and the bracket is the bracket of vector fields.
\end{definition}

\subsubsection{The derived algebraic structure} To any negatively graded differential graded Lie algebra  there is an associated dg-Leibniz algebra (see
\cite{Uchino} for the explicit construction and also \cite{KosmannDerived, Kosmann, VoronovDerived1, VoronovDerived2, GetzlerDerived}).

\begin{definition}
The derived dg-Leibniz algebra ${\rm{D}}{\gr{sym}}^*(P)$ of
${\gr{sym}}^*(P,Q)$ is the complex $${\rm{D}}{\gr{sym}}^*(P,Q)
:={\gr{sym}}^{* < 0}(P,Q)[1]$$ together with the differential
$\delta:=[Q_P,\_]$ and the derived bracket
$$\lf a,b \rf := (-1)^{\| a\|}[[Q_P,a],b].$$
\end{definition}

It is a fact that ${\rm{D}}{\gr{sym}}^*(P,Q)$
becomes a dg-Leibniz algebra; namely that $\delta$ and $\lf , \rf$
satisfy the properties
\begin{eqnarray*}
\delta \lf a,b \rf  &=& \lf \delta a, b\rf + (-1)^{\| a \|}\lf a, \delta b\rf\\
\lf a,\lf b,c \rf \rf &=& \lf \lf a,b \rf, c \rf + (-1)^{\| a \|
\| b\|} \lf b, \lf a,c \rf \rf
\end{eqnarray*}
where $\|a\|$ denotes the degree of $a$ in
${\rm{D}}{\gr{sym}}^*(P,Q)$ and, therefore, $\|a\| = |a|+1$ where
$|a|$ is the degree of $a$ in ${\gr{sym}}^*(P,Q)$.

Note that there is a canonical graded action
\begin{align*}
{\gr{sym}}^*(P,Q) \times {\rm{D}}{\gr{sym}}^*(P,Q) & \to {\rm{D}}{\gr{sym}}^*(P,Q)\\
a \cdot b & \mapsto [a,b]
\end{align*}
whose properties follow from the fact that $[Q,\_]$ and $[,]$ form a differential graded Lie algebra. In particular, note that
since the elements $a \in {\gr{sym}}^0(P,Q)$ satisfy $[Q,a]=0$, then ${\gr{sym}}^0(P,Q)$ acts by derivations
on the Leibniz algebra $ ({\rm{D}}{\gr{sym}}^0(P,Q), \lf, \rf)$.

\subsection{Symmetries of $\real[n]$-bundles over $T[1]M$.}
Let $\RR=(T[1]M \oplus \real[n], Q_\RR=d + \Theta \partial_t)$ be a $\real[n]$-bundle over $\MM$ where $\Theta \in \Omega^{n+1}_{\rm{cl}}M$
and $t$ is a variable of degree $n$.
Since the bracket of $Q_\RR$ with a generic degree $0$ vector field gives
$$[d + \Theta \partial_t, \LL_X + B \partial_t]= ( dB - \LL_X \Theta ) \partial_t$$ then the degree $0$
symmetries of $\RR$ is the Lie algebra
$${\gr{sym}}^0(\RR,Q_\RR) = \{\LL_X + B \partial_t | X \in \gr XM, B \in \Omega^nM \ \mbox{and} \ (dB - \LL_X \Theta)=0 \},$$
and the negatively graded symmetries are
\begin{align*}
{\gr{sym}}^{-1}(\RR,Q_\RR) &= \{\iota_Y + A \partial_t | Y \in \gr XM  \ \mbox{and} \ A \in \Omega^{n-1}M \}\\
{\gr{sym}}^{-k}(\RR,Q_\RR) &= \{ \eta \partial_t | \eta \in \Omega^{n-k}M \} \ \ \mbox{for} \ \ k > 1.
\end{align*}

The differential in ${\gr{sym}}^*(\RR,Q_\RR)$ becomes
\begin{align*} 
[ Q_\RR, \LL_X + B\partial_t ]   &= 0 \\
[ Q_\RR, \iota_X + A \partial_t ]   & =\LL_X + (d A+\iota_X\Theta ) \partial_t \nonumber \\
[  Q_\RR, \eta \partial_t  ]    &=   (d \eta) \partial_t, \nonumber
\end{align*}
 and the brackets become
 \begin{align*}
 [\LL_{X_0} + B_0 \partial_t, \LL_{X_1} + B_1\partial_t]  &= \LL_{[X_0,X_1]} + (\LL_{X_0}B_1 - \LL_{X_1}B_0)\partial_t\\
 [\LL_X + B \partial_t, \iota_Y +A \partial_t]  & =\iota_{[X,Y]} + (\LL_XA - \iota_YB)\partial_t\\
 [\LL_X + B \partial_t, \eta \partial_t] &=  (\LL_X \eta)\partial_t\\
 [\iota_{Y_0} + A_0 \partial_t, \iota_{Y_1} + A_1 \partial_t]  & = (\iota_{Y_0}A_1 + \iota_{Y_1}A_0)\partial_t\\
  [\iota_X + \alpha \partial_t, \eta \partial_t] & =(\iota_X \eta) \partial_t.
 \end{align*}

 Note that, when the $n+1$ form $\Theta=0$, the differential graded Lia algebra structure defined
 above is the same one that was defined by Dorfman in \cite{Dorfman}.
 
 \subsection{Derived symmetries of $\real[n]$-bundles over $T[1]M$}
 The derived dg-Leibniz algebra of the $\real[1]$-bundle $\RR$ over $\MM$ is
$$ {\rm{D}}{\gr{sym}}^k(\RR,Q_\RR) \cong \left\{
\begin{array}{ccr}
\gr X M \oplus \Omega^{n-1}M & {\rm{if}} & k=0 \\
\Omega^{n-1-k}M & {\rm{if}} & k <0
\end{array} \right.$$
where the differential $[Q_\RR, \_]$ becomes the De Rham differential
$$\Omega^0M \stackrel{d}{\To} \cdots \stackrel{d}{\To}
\Omega^{n-2}M \stackrel{d}{\To} \gr X M \oplus \Omega^{n-1}M,$$
and the brackets are given by the formulas
\begin{eqnarray*}
\lf \iota_{X_0} + A_0 \partial_t , \iota_{X_1} +  A_1 \partial_t \rf
  & = & \iota_{[X_0,X_1]} + \left( \LL_{X_0} A_1 - \iota_{X_1} dA_0 -\iota_{X_1}
\iota_{X_0} \Theta \right) \partial_t\\
\lf \iota_X + \alpha \partial_t , \eta \partial_t \rf & = & \LL_X
\eta \partial_t \\
\lf \eta \partial_t,  \iota_X+A \partial_t\rf &=& -(\iota_X d \eta) \partial_t\\
 \lf \mu \partial_t, \eta \partial_t \rf &=& 0.
\end{eqnarray*}

\subsubsection{Derived symmetries of $\real[1]$-bundles}
The derived symmetries of $(T[1]M \oplus \real[1], d + F \partial_q)$ is simply the Lie algebra
$(\gr X M \oplus C^\infty M, \lf,\rf)$ where the bracket is
$$\lf X \oplus f, Y \oplus g \rf = [X,Y] \oplus (X(g)-Y(f) - \iota_X \iota_YF).$$
Whenever $F$ is the curvature 2-form of a principal $S^1$-bundle $E \to M$; this Lie algebra
is isomorphic to the Lie algebra $\gr X E^{S^1}$ of invariant vector fields on $E$.

\subsubsection{Derived symmetries of $\real[2]$-bundles}
In the case that $n=2$ and whenever $\Theta=H$ is  a closed three form, the derived
algebra of $(T[1]M \oplus \real[2], d + H \partial_t)$ becomes the complex
$$ \Omega^0 M \stackrel{d}{\To} \gr X M \oplus \Omega^1 M$$
where the bracket $\lf, \rf$ is precisely the $H$-twisted
Courant-Dorfman bracket of the exact Courant algebroid $TM \oplus
T^*M$ (cf. \cite{Bursztyn, Roytenberg}).

\subsection{Symmetries of $\real[2]$-bundles over $\real[1]$-bundles}
For the $\real[2]$-bundle over the $\real[1]$-bundle 
$$\PP:= ((T[1]M \oplus \real[1]) \oplus \real[2], Q_\PP=d + F \partial_q +(H +  q\bar{F} ) \partial_t )$$
as in chapter \ref{section R2-R1-bundles}, we have that its differential graded Lie algebra of symmetries ${\gr{sym}}^*(\PP, Q_\PP)$ is defined by 
\begin{align*}
{\gr{sym}}^{* < -2}(\PP,Q_\PP)&=0\\
{\gr{sym}}^{-2}(\PP,Q_\PP)&\cong (C^\infty M)\partial_t\\
{\gr{sym}}^{-1}(\PP,Q_\PP)&\cong \gr X M \oplus (C^\infty M)\partial_q \oplus ( \Omega^1M \oplus C^\infty M)\partial_t \\
{\gr{sym}}^0(\PP,Q_\PP)&=\{ \scr X \in {\rm Vect}^{0}(\PP) | [Q_\PP, \scr X]=0 \}\\
\end{align*}
with differential $[Q_\PP, \_]$ 
$$0 \to {\gr{sym}}^{-2}(\PP,Q_\PP) \stackrel{[Q_\PP, \_]}{\To} {\gr{sym}}^{-1}(\PP,Q_\PP) \stackrel{[Q_\PP, \_]}{\To} {\gr{sym}}^{0}(\PP,Q_\PP)  $$
and with bracket $[,]$. 

The elements in ${\gr{sym}}^{0}(\PP, Q_\PP)$, which are the degree zero derivations that commute with $Q_\PP$, are the elements 
$$ \scr X = \LL_X +A \partial_q +(B +  q \bar{A} )\partial_t  \in {\rm Vect}^{0}(\PP)$$
such that $[Q_\PP, \scr X]=0$, and this is equivalent to the equations
\begin{align*}
dA -\LL_X F =&0 \\
 dB -\LL_X H +  F \wedge \bar{A} -  A \wedge \bar{F} =&0\\
  d\bar{A} + \LL_X \bar{F}=&0.
\end{align*}

 \subsubsection{The differential structure}
Vector fields of degree $-2$ will be denoted by
$h \partial_t $
for $h$ in $C^\infty M$ and vector fields of degree $-1$ will be denoted by
$$\iota_Y +f \partial_q +(C + q\bar{f}) \partial_t$$
with $\iota_Y$ the contraction on the vector field $Y$ of $M$,$f,\bar{f}$ in $C^\infty M$ and $C$ in $\Omega^1M$.
Then $[Q_\PP, \_]$  acts as follows:
\begin{align*}
[Q_\PP, h \partial_t  ] = & dh \partial_t\\
[Q_\PP, \iota_Y +f \partial_q +(C + q\bar{f}) \partial_t ] = & 
 \LL_Y +(df  + \iota_YF ) \partial_q \\ + &(dC +\iota_YH +  F\bar{f} +f \bar{F})\partial_t\\
  - & q(d\bar{f}+ \iota_Y\bar{F})\partial_t\\
 [Q_\PP, \LL_X +A \partial_q +(B +  q \bar{A} )\partial_t ] = & 0
\end{align*}

\subsubsection{The graded Lie algebra structure}
The graded Lie algebra structure of ${\gr{sym}}^{*}(\PP,Q_\PP)$ is given
by the following explicit formulas.

For two degree zero derivations, the bracket is:
\begin{align*}
 [\LL_{X_0} +A_0 \partial_q +(B_0 +  q \bar{A}_0 )\partial_t, \LL_{X_1} +A_1 \partial_q  &+(B_1 +   q \bar{A}_1 )\partial_t ] =  \\
 &\LL_{[X_0,X_1]} \\ + & (\LL_{X_0} A_1-\LL_{X_1} A_0)\partial_q  \\
+ & (\LL_{X_0} B_1-\LL_{X_1} B_0 +  A_0\bar{A}_1 - A_1\bar{A}_0)\partial_t \\+  &   q(  \LL_{X_0} \bar{A}_1   -\LL_{X_1} \bar{A}_0)\partial_t
 \end{align*}
and note that when $X_0=0=X_1$ the bracket simplifies to 
$$[A_0 \partial_q +(B_0 +  q \bar{A}_0 )\partial_t, A_1 \partial_q +(B_1 +  q \bar{A}_1 )\partial_t ]=(A_0\bar{A}_1 - A_1\bar{A}_0)\partial_t.$$

For two degree -1 derivations, the bracket is:
\begin{align*}
 [\iota_{Y_0} +f_0 \partial_q +(C_0 +  q \bar{f}_0 )\partial_t, \iota_{Y_1} +f_1 \partial_q +(C_1 +  q \bar{f}_1 )\partial_t ] &= \\
 (\iota_{X_0} C_1  & + \iota_{X_1} C_0 +f_0\bar{f}_1 + f_1\bar{f}_0)\partial_t
 \end{align*}
and the remaining two brackets are
\begin{align*}
 [\LL_{X} +A \partial_q +(B +  q \bar{A} )\partial_t,\iota_{Y} +f \partial_q +(C + & q \bar{f} )\partial_t ] = \\
 &\iota_{[X,Y]} \\+ & (\LL_{X} f-\iota_{Y} A)\partial_q  \\+& 
 (\LL_{X} C-\iota_{Y} B +   A\bar{f}  - f\bar{A} )\partial_t \\+ &   q( \LL_{X} \bar{f}   +\iota_{Y} \bar{A})\partial_t \\
 [\LL_{X} +A \partial_q +(B +  q \bar{A})&\partial_t, h   \partial_t]= \LL_X h \partial_t
 \end{align*}

\subsubsection{The derived dg-Leibniz algebra}

The derived algebra is then
$$ {\rm{D}}{\gr{sym}}^k(\PP, Q_\PP) \cong \left\{
\begin{array}{ccl}
\gr X M \oplus (C^\infty M)\partial_q \oplus (\Omega^{1}M \oplus q(C^\infty M))\partial_t  & {\rm{if}} & k=0 \\
(C^\infty M)\partial_t & {\rm{if}} & k =-1
\end{array} \right.$$
where the differential $[Q_\PP, \_]$ becomes the De Rham differential, and the derived bracket is given by the formula
\begin{align} \label{Courant bracket in R[2]-R[1]}
  &\lf \iota_{X} +f_0 \partial_q +(C_0 +  q \bar{f}_0 )\partial_t,   \iota_{Y} +f_1 \partial_q +(C_1 +  q \bar{f}_1 )\partial_t \rf
   =   \\ &  \iota_{[X,Y]} \nonumber \\+&   \left( \LL_X f_1 - \iota_Y df_0 -\iota_Y
\iota_X F \right)  \partial_q \nonumber  \\+ &
(\LL_X C_1 - \iota_Y dC_0 -\iota_Y
\iota_X H - \iota_Y(F \bar{f}_0) - \iota_Y(f_0\bar{F} ) +(df_0) \bar{f}_1 + (\iota_XF) \bar{f}_1 + f_1d\bar{f}_0 +f_1 \iota_X\bar{F} )\partial_t \nonumber \\ + & q( \LL_X \bar{f}_1 -  \iota_Y d\bar{f}_0 -\iota_Y
\iota_X \bar{F} )  \partial_t \nonumber 
\end{align}
and the extra action is given by the bracket
\begin{align*}
\lf \iota_X + f\partial_q +(C +q\bar{f}) \partial_t , h \partial_t \rf & =  \LL_X
h \partial_t .
\end{align*}

\begin{rem} \label{remark courant bracket}
Since $(H + q \bar{F})$ is a closed 3-form on the dg-manifold $\EE =(T[1]M \oplus \real[1], Q_\EE=d+ F \partial_q)$, the previous
derived algebraic structure is equivalent to the derived algebraic structure associated to $(\EE \oplus \real[2],Q_\EE = (H + q \bar{F})\partial_t)$ where the Courant-Dorfman bracket has the usual structure
\begin{align*} 
  \lf (\iota_{X} +f_0 \partial_q) +(C_0 +  q \bar{f}_0 )\partial_t,   (\iota_{Y} +f_1 \partial_q) & +(C_1 +  q \bar{f}_1 )\partial_t \rf
   =  \\
&    \iota_{[X,Y]} +(X(f_1) - Y(f_0) - \iota_Y\iota_X F) \partial_q \\
   + & [Q_\EE, \iota_{X} +f_0 \partial_q] (C_1 +  q \bar{f}_1 )\partial_t \\
   - & (\iota_{Y} +f_1 \partial_q) \left( Q_\EE (C_0 +  q \bar{f}_0 ) \right) \partial_t \\
   - & (\iota_{Y} +f_1 \partial_q) (\iota_{X} +f_0 \partial_q) (H + q \bar{F})\partial_t,
\end{align*}
where we should think of $Q_\EE$ as the De Rham differential and of $[Q_\EE, \iota_{X} +f_0 \partial_q]$ as the Lie derivative.
\end{rem}

\subsection{Isomorphism of symmetries for T-dual pairs} \label{Isomorphism of symmetries for T-dual pairs}
Consider a T-dual pair of $\real[2]$-bundles $(\EE, \PP)$ and $(\bar{\EE}, \bar{\PP})$ with
\begin{align*}
{\PP}&=((T[1]M \oplus \real[1])\oplus \real[2], Q_{{\PP}}=d + {F}\partial_{{q}} + (H + {q} \bar{F})\partial_t)\\
\bar{\PP}&=((T[1]M \oplus \real[1])\oplus \real[2], Q_{\bar{\PP}}=d + \bar{F}\partial_{\bar{q}} + (H + \bar{q} F)\partial_t)
\end{align*}
as in section \ref{subsection T-duality}. Consider the differential graded Lie algebras of symmetries ${\gr{sym}}^{*}(\PP, Q_\PP)$ and
${\gr{sym}}^{*}(\bar{\PP}, Q_{\bar{\PP}})$ of both dg-manifolds and define the isomorphism of graded vector spaces
\begin{align*}
\Phi: {\gr{sym}}^{*}(\PP, Q_\PP) & \to {\gr{sym}}^{*}(\bar{\PP}, Q_{\bar{\PP}})\\
\LL_{X} +A \partial_q +(B +  q \bar{A} )\partial_t & \mapsto \LL_{X} -\bar{A} \partial_{\bar{q}} +(B -  \bar{q} {A} )\partial_t\\
\iota_Y +f \partial_q +(C + q\bar{f}) \partial_t & \mapsto \iota_Y +\bar{f} \partial_{\bar{q}} +(C + \bar{q}{f}) \partial_t\\
h \partial_t & \mapsto h \partial_t.
\end{align*}
A straightforward calculation shows that indeed the map $\Phi$ is an isomorphism of differential graded Lie algebras.

\begin{theorem} \label{theorem T-duality iso for derived symmetries}
 Let $(\EE, \PP)$ and $(\bar{\EE}, \bar{\PP})$ be T-dual dg-manifolds over $\MM$. Then
the map $$\Phi: {\gr{sym}}^{*}(\PP, Q_\PP)  \to {\gr{sym}}^{*}(\bar{\PP}, Q_{\bar{\PP}})$$ is an isomorphism of differential graded
Lie algebras. In particular, the induced map $${\rm D}\Phi: {\rm D}{\gr{sym}}^{*}(\PP, Q_\PP)  \to {\rm D}{\gr{sym}}^{*}(\bar{\PP}, Q_{\bar{\PP}})$$
on the derived dg-Leibniz algebras is also an isomorphism.
\end{theorem}
The previous theorem implies the isomorphism of exact algebroids associated to twisted T-dual $S^1$-principal bundles that was
proved in \cite{CavalcantiGualtieri}. Now we proceed to clarify this point.
\subsubsection{T-duality isomorphism for exact Courant algebroids}
Consider the $S^1$-principal bundles $E \stackrel{\pi}{\to} M$ and $\bar{E} \stackrel{\bar{\pi}}{\to} M$ with associated 
curvature closed 2-forms $F$ and $\bar{F}$ respectively, and consider $S^1$-invariant and closed 3-forms 
$\eta \in (\Omega^3_{\rm{cl}} E)^{S^1}$ and $\bar{\eta} \in (\Omega^3_{\rm{cl}} \bar{E})^{S^1}$ such that
$\eta = \pi^*H + \theta \wedge \pi^*\bar{F}$ and $\bar{\eta} = \bar{\pi}^*H + \bar{\theta} \wedge \bar{\pi}^*F$ where  $\theta$ and $\bar{\theta}$ 
are connection 1-forms on $E$ and $\bar{E}$ 
respectively.

The choice of connections provides us with isomorphisms $(TE)/{S^1} \cong TM \oplus \langle \partial_\theta \rangle$, 
$(T\bar{E})/{S^1} \cong TM \oplus \langle \partial_{\bar{\theta}} \rangle$, $(T^*E)/{S^1} \cong T^*M \oplus \langle \theta \rangle$
and $(T^*\bar{E})/{S^1} \cong T^*M \oplus \langle \bar{\theta} \rangle$ where $\partial_\theta$ and $\partial_{\bar{\theta}}$ denote the vector fields generated by the circle action on $E$ and $\bar{E}$  respectively with period 1. 
A section of $(TE \oplus T^*E)/S^1$ can  be written as 
$$X + f \partial_\theta + C + \bar{f}  \theta$$
where $X$ is a vector field on $M$, $f, \bar{f}$ are functions on $M$ and $C$ is a 1-form on $M$ and, therefore, we can write
the isomorphism \begin{align*}
\Psi: \Gamma(TE \oplus T^*E)^{S^1} & \to  {\gr{sym}}^{0}(\PP, Q_\PP)\\
X + f \partial_\theta + C + \bar{f}  \theta & \mapsto \iota_X + f \partial_q + (C + q\bar{f})\partial_t
\end{align*}
which, in view of Remark \ref{remark courant bracket}, it induces an isomorphism of Leibniz algebras
$$\Psi:( \Gamma(TE \oplus T^*E)^{S^1}, [,]_\eta)  \stackrel{\cong}{\to}  ({\gr{sym}}^{0}(\PP, Q_\PP),\lf,\rf )$$
where $ [,]_\eta$ is the $\eta$-twisted Courant-Dorfman bracket.
Obtaining the equivalent isomorphism for the dual
\begin{align*}
\bar{\Psi}: (\Gamma(T\bar{E} \oplus T^*\bar{E})^{S^1}, [,]_{\bar{\eta}}) & \stackrel{\cong}{\to}  ({\gr{sym}}^{0}(\bar{\PP}, Q_{\bar{\PP}}), \lf,\rf)\\
X + f \partial_{\bar{\theta}} + C + \bar{f}  \bar{\theta} & \mapsto \iota_X + f \partial_{\bar{q}} + (C + \bar{q}\bar{f})\partial_t
\end{align*}
we conclude that
\begin{cor} \label{corollary T-duality for Courant algebroids}
For the T-dual pair of principal $S^1$-bundles $(E,\eta)$ and $(\bar{E}, \bar{\eta})$ over $M$ defined above,
the composition
\begin{align*}
 \bar{\Psi}^{-1} \circ \Phi \circ \Psi: (\Gamma(T{E} \oplus T^*{E})^{S^1} &\to \Gamma(T\bar{E} \oplus T^*\bar{E})^{S^1}\\
X + f \partial_{{\theta}} + C + g  {\theta}  & \mapsto X + g \partial_{\bar{\theta}} + C + f  \bar{\theta} 
\end{align*}
induces an isomorphism of Courant algebroids
$$(\Gamma(T{E} \oplus T^*{E})^{S^1}, [,]_{{\eta}}) \cong (\Gamma(T\bar{E} \oplus T^*\bar{E})^{S^1}, [,]_{\bar{\eta}}).$$
\end{cor}

\subsection{Relation with exceptional generalized geometry} Exceptional generalized geometry is an algebraic framework
suited to the study of solutions of M-theory with fluxes \cite{Hull, PachecoWaldram}. In \cite{Baraglia}, it is shown how the framework
of the exceptional generalized geometry can be obtained from certain properties associated to simple Lie algebras. Rather than reproducing 
what has been done in the above cited references, we will show how the symmetries of $\real[n]$-bundles over 
dg-manifolds provide an alternative way to obtain some of the algebraic structures that appear in \cite{Baraglia}.

\subsubsection{$B_n$ generalized geometry}
Consider a $\real[2]$-bundle over a $\real[1]$-bundle over $\MM$ which is T-dual to itself, that is
$$\PP=((T[1]M \oplus \real[1])\oplus \real[2], d + F \partial_q + (H + q F )\partial_t),$$
satisfying the equations $dF=0$ and $dH + F \wedge F=0.$

Consider the sub-differential graded Lie algebra $\BB^*$ of ${\gr{sym}}^{*}(\PP, Q_\PP) $ which consist of elements of the form
\begin{align*}
\LL_X + A\partial_q + (B -qA) \partial_t & \in \BB^0\\
\iota_Y + f \partial_q + (C + qf) \partial_t & \in \BB^{-1}\\
h \partial_t & \in \BB^{-2}
\end{align*}
and note that the restrictions on $A$ and $B$ become the equations
$$dA -\LL_XF =0 \ \ \mbox{and} \ \ dB - \LL_XH = 0.$$
The Lie bracket on $\BB^0$ becomes 
\begin{align*}
[\LL_{X_0} + A_0\partial_q + (B_0 -qA_0) \partial_t, \LL_{X_1} + A_1\partial_q & + (B_1-qA_1) \partial_t]= \\
& \LL_{[X_0,X_1]} + (\LL_{X_0}A_1 - \LL_{X_1}A_0) \partial_q \\
+& (\LL_{X_0}B_1 - \LL_{X_1}B_0 +2A_0A_1)\partial_t\\
-&q( \LL_{X_0}A_1 - \LL_{X_1}A_0) \partial_t
\end{align*}
and, therefore, $\BB^0$ fits in the middle of the short exact sequence
 $$ 0 \to \Omega^1_{\rm{cl}}M \oplus  \Omega^2_{\rm{cl}}M \to \BB^0 \to \gr XM \to 0.$$

If we denote 
$$(X,f,C):=\iota_{X} +f \partial_q +(C +  q f )\partial_t$$
then the derived bracket on ${\rm{D}}\BB^0$ becomes
\begin{align*}
\lf (X,f,C),(Y,g,D) \rf  = 
 ([X, & Y], (\LL_X g  - \LL_Yf -\iota_Y \iota_XF ),\\
( & \LL_X D -  \iota_Y dC -\iota_Y 
\iota_X H - 2f\iota_YF    + 2g\iota_XF + 2gdf ))
\end{align*}
and the pairing comes from the bracket of degree -1 derivations, i.e.
\begin{align*}
\langle (X,f,C),(Y,g,D) \rangle = \iota_X D + \iota_Y C +2 fg.
\end{align*}
The Lie algebra $\BB^0$ act on ${\rm{D}}\BB^0$ by the bracket $[,]$ and we see that closed 1-forms $A$
and closed 2-forms $B$ act as follows
\begin{align*}
A (X,f,C) &= (0, -\iota_XA, 2Af)\\
B (X,f,C) &= (0, 0, -\iota_XB ).
\end{align*}

Note furthermore that since the $\real[2]$-bundle $\PP$ is T-dual to itself, the automorphism of differential graded Lie algebras
$\Phi: {\gr{sym}}^{*}(\PP, Q_\PP)  \to {\gr{sym}}^{*}({\PP}, Q_{{\PP}})$ defined in section \ref{Isomorphism of symmetries for T-dual pairs}
induces an automorphism ${\rm D}\Phi:{\rm D} {\gr{sym}}^{*}(\PP, Q_\PP)  \to {\rm D}{\gr{sym}}^{*}({\PP}, Q_{{\PP}})$ on 
the derived dg-Leibniz algebras. It is easy to see that the dg-Leibniz algebra  ${\rm{D}}\BB^*$ is precisely the fixed points
of the automorphism ${\rm D}\Phi$. 
Finally, noting  that the previous algebraic structure on $\gr XM \oplus C^\infty M \oplus \Omega^1M$ is 
the structure underlying the $B_n$ generalized geometry that is defined in section 2.4 of \cite{Baraglia}, we conclude

\begin{theorem} \label{theorem Bn geometry}
The $B_n$ generelized geometry structure on $\gr XM \oplus C^\infty M \oplus \Omega^1M$ defined in section 2.4 of \cite{Baraglia}
twisted by the closed 2-form $F$ and the 3-form $H$ satisfying $dH + F \wedge F=0$, 
is isomorphic to the fixed sub dg-Leibniz algebra of fixed points of the T-duality automorphism 
of Courant-Dorfman algebroids
$${\rm D}\Phi:{\rm D} {\gr{sym}}^{*}(\PP, Q_\PP)  \to {\rm D}{\gr{sym}}^{*}({\PP}, Q_{{\PP}})$$
on the derived dg-Leibniz algebra of the symmetries of the self T-dual dg-manifold
$$\PP=((T[1]M \oplus \real[1])\oplus \real[2], d + F \partial_q + (H + q F )\partial_t).$$
\end{theorem}

In particular, we note that the algebraic structure underlying the
 $B_n$ generalized geometry can be obtained as a sub-algebra of the algebraic structure underlying
 generalized geometry.

\subsubsection{$E_6$ generalized geometry} \label{section E6 generalized}
Consider a $\real[6]$-bundle over a $\real[3]$-bundle over $\MM$ which is given by the data
$$\PP=((T[1]M \oplus \real[3])\oplus \real[6], d + F_4 \partial_q + (F_7 + q \frac{F_4}{2} )\partial_t)$$
for $q$ a variable of degree 3, $t$ a variable of degree 6,
$F_4$ a closed 4-form and $F_7$ a 7-form satisfying the equation $dF_7 + \frac{1}{2}F_4 \wedge F_4=0.$

Consider the sub-differential graded Lie algebra $\CC^*$ of ${\gr{sym}}^{*}(\PP, Q_\PP) $ which consist of elements of the form
\begin{align*}
\LL_X + A_3\partial_q + (B_6 -q\frac{A_3}{2}) \partial_t & \in \CC^0\\
\iota_Y + \sigma_2 \partial_q + (\sigma_5 + q\frac{\sigma_2}{2}) \partial_t & \in \CC^{-1}\\
\eta_1 \partial_q + (C_4 -q\frac{\eta_1}{2}) \partial_t & \in \CC^{-2}\\
f \partial_q + (D_3 + q\frac{f}{2}) \partial_t & \in \CC^{-3}\\
h \partial_t & \in \CC^{< -3}
\end{align*}
and note that the restrictions on $A_3$ and $B_6$ become the equations
$$dA_3 -\LL_XF_4 =0 \ \ \mbox{and} \ \ dB_6 - \LL_XF_7 = 0.$$

If we denote 
$$(X,\sigma_2, \sigma_5):=\iota_{X} +\sigma_2\partial_q +(\sigma_5 +  q \frac{\sigma_2}{2} )\partial_t$$
then the derived bracket on ${\rm{D}}\CC^0$ becomes
\begin{align*}
\lf (X,\sigma_2, \sigma_5),(Y,\tau_2, \tau_5) \rf  = 
 ([X, & Y], (\LL_X \tau_2  - \iota_Y \sigma_2 -\iota_Y \iota_XF_4 ),\\
(  \LL_X \tau_5 -  \iota_Y d\sigma_5  -\iota_Y 
\iota_X F_7  &- \iota_Y(F_4 \wedge \sigma_2)    + (\iota_XF_4 )\wedge\tau_2  + d\sigma_2 \wedge \tau_2 ))
\end{align*}
and the pairing becomes 
\begin{align*}
\langle (X,\sigma_2, \sigma_5),(Y,\tau_2, \tau_5) \rangle =(0, \iota_X \tau_2 + \iota_Y \sigma_2,
 \iota_X \tau_5 + \iota_Y \sigma_5 +\sigma_2 \wedge \tau_2 ).
\end{align*}
The Lie algebra $\CC^0$ act on ${\rm{D}}\CC^0$ by the bracket $[,]$ and we see that closed 3-forms $A_3$
and closed 6-forms $B_6$ act as follows
\begin{align*}
A_3 (X,\sigma_2, \sigma_5) &= (0, -\iota_XA_3, A_3 \wedge \sigma_2)\\
B_6 (X,\sigma_2, \sigma_5) &= (0, 0, -\iota_XB_6 ).
\end{align*}

The dg-Leibniz algebra ${\rm{D}}\CC^*$ just defined is one piece of the algebraic structure
defined on the exceptional generalized geometry framework \cite{Hull, PachecoWaldram}. The 
previous algebraic structure on $\gr X M \oplus \Omega^2 M \oplus \Omega^5 M$ is the underlying
algebraic structure of the $E_6$ generalized geometry outlined in \cite[Chapter 11]{Baraglia}
\subsubsection{} We conclude this section by noting that the category of dg-manifolds provides a natural framework
for generalized geometry and for its exceptional relative. In particular, we note that the derived algebraic 
structure obtained in the previous 
two sections satisfies the axioms of a dg-Leibniz algebra for they are constructed as derived
differential graded Lie algebras. The derived brackets in section \ref{section E6 generalized} may appear complicated
at first sight, nevertheless they are the brackets that satisfy the desired properties (cf. \cite[Chapter 11]{Baraglia}).

\end{document}